\documentclass[review,1p]{elsarticle}

\usepackage{amssymb}
\usepackage{amsmath}
\usepackage{lyu}
\usepackage{algorithm}
\usepackage{algpseudocode}
\usepackage{amsthm}\mathchardef\mhyphen="2D
\newtheorem{theorem}{Theorem}

\theoremstyle{definition}

\theoremstyle{remark}
\newtheorem{remark}[theorem]{Remark}
\newtheorem{proposition}{Proposition}

\journal{Journal of Computational Physics}

\begin{document}

\begin{frontmatter}



\title{A stochastic branching particle method for solving non-conservative reaction--diffusion equations}


\author[1,2]{Liyao Lyu}
\ead{lyuliyao@math.ucla.edu}
\author[2,3]{Huan Lei}
\ead{leihuan@msu.edu}

\address[1]{{Department of Mathematics, University of California, Los Angeles}, 520 Portola Plaza, Los Angeles, CA, USA}

\address[2]{{Department of Computational Mathematics, Science \& Engineering, Michigan State University}, 428 S Shaw Ln, East Lansing, MI, USA}

\address[3]{{Department of Statistics \& Probability, Michigan State University}, 619 Red Cedar Road, East Lansing, MI, USA}

\begin{abstract}
We propose a stochastic branching particle-based method for solving nonlinear non-conservative advection–diffusion–reaction equations. 
The method splits the evolution into an advection–diffusion step, based on a linearized Kolmogorov forward equation and approximated by stochastic particle transport, and a reaction step implemented through a branching birth–death process that provides a consistent temporal discretization of the underlying reaction dynamics. This construction yields a mesh-free, nonnegativity-preserving scheme that naturally accommodates non-conservative systems and remains robust in the presence of singularities or blow-up. 
We validate the method on two representative two-dimensional systems: the Allen–Cahn equation and the Keller–Segel chemotaxis model. In both cases, the present method accurately captures nonlinear behaviors such as phase separation and aggregation, and achieves reliable performance without the need for adaptive mesh refinement.

\end{abstract}




\end{frontmatter}



\section{Introduction}
The advection--diffusion--reaction (ADR) type partial differential equations (PDEs) have attracted much interest in physics, biology, and engineering for modeling various phenomena relevant to chemical reactions, transport, and collective behaviors in biological systems \cite{hundsdorfer2013numerical}. Traditional grid-based methods, such as finite difference~\cite{hoff1978stability}, finite element~\cite{codina1998comparison}, and discontinuous Galerkin schemes~\cite{ayuso2009discontinuous,xia2009application}, have achieved great success in solving many of these problems. However, from a numerical perspective, an essential challenge in solving such PDEs lies in handling nonlinearities and singular structures in certain cases, e.g., the aggregation phenomena of the Keller--Segel (KS) model \cite{keller1970initiation} and the interfacial dynamics of the Allen--Cahn (AC) model \cite{Allen_Cahn_1979}.  
Mesh-based discretizations approximate solutions in function spaces that may not be well-suited for capturing the blow-up or highly localized behavior, often resulting in reduced accuracy or numerical instabilities, and requiring additional numerical treatments.

Particle-based approaches offer a natural alternative by representing the solutions as empirical measures, which are particularly effective for problems exhibiting concentration or blow-up phenomena. 
Existing particle methods, such as the particle-in-cell method~\cite{birdsall2018plasma, hockney2021computer} for solving the Vlasov--Poisson equation, are often based on well-established  underlying physical laws. In this setting, particle dynamics can be derived from the explicit micro-scale governing equations. In contrast, the explicit physical laws are often unknown for general non-conservative ADR type PDEs, where effective interactions may be highly nonlinear, system-specific, or even not well characterized. For example, in many biological and chemical systems such as chemotaxis, the micro-scale interactions remain poorly understood, which limits the applicability of traditional particle-based approaches.
To extend particle-based methods beyond physically motivated settings, blob methods~\cite{craig2016blob,carrillo2019blob,huang2017error} and semi-Lagrangian particle schemes~\cite{qiu2010conservative,qiu2011conservative,qiu2011positivity} have been developed for conservative PDEs.
These approaches are based on the mean-field perspective, in which the empirical measures of interacting particle systems converge to the solution of the original PDE. This mathematical framework establishes a rigorous link between the particle dynamics and the continuum equations and does not rely on explicit microscale physical laws. Despite these advances, there is still no unified particle-based framework for systematically discretizing nonlinear and non-conservative PDEs, especially in the presence of singularities, sharp interface or blow-up phenomena.

Machine learning-based numerical PDE algorithms, such as Deep Galerkin Methods~\cite{sirignano2018dgm}, Deep Ritz method~\cite{yu2018deep}, and Mixed Residual method~\cite{lyu2022mim}, have attracted significant attention in recent years. By leveraging the strong representability of neural networks, these methods are particularly promising for solving high-dimensional problems. In the context of particle-based approaches, coupling neural networks with the Feynman-Kac formula~\cite{han2024deep, hutzenthaler2021speed, nguwi2024deep,beck2021deep} or backward stochastic differential equations (SDEs) ~\cite{han2017deep, han2018solving, hure2020deep} has also emerged as a powerful strategy for solving a broad class of PDEs. Nevertheless, the approximation capability of neural networks is fundamentally grounded in the universal approximation theorem that targets continuous solutions on compact sets.
As a result, neural-network–based solvers may face intrinsic difficulties when dealing with nonlinear, non-conservative PDEs that develop singularities or blow-up phenomena, where discontinuities or concentration effects play a central role.

In this work, we propose a \textit{stochastic branching particle method} to address these challenges. 
At each time step, the original PDE is locally linearized and approximated by the Kolmogorov forward equation, which establishes a direct connection between the PDE and the stochastic particle dynamics.  Unlike approaches that rely on explicit physical laws or mean-field limit arguments, this framework establishes a systematic particle-based discretization applicable to general nonlinear and non-conservative PDEs. 
A key feature of the method is a particle-level birth–death mechanism, introduced to represent non-conservative reaction terms by allowing the number of particles to evolve consistently with the underlying reaction dynamics. The only step that invokes a continuity assumption is the interpolation of the empirical measure, which can be replaced by blob- or kernel-based representations. Consequently, the developed method does not fundamentally rely on continuity and is well-suited for problems with singularities or blow-up phenomena. 
From a numerical perspective, the proposed method is broadly compatible with existing stochastic interpretations of reaction–diffusion systems. For example, in the case of chemotaxis, it aligns with the mean-field derivation of the KS equation from moderately interacting particles developed by Stevens~\cite{Stevens_SIAM_2000}. However, our formulation is not restricted to this setting and applies generally to nonlinear ADR systems with non-conservative terms.


To illustrate the effectiveness of the proposed method, we consider two representative non-conservative reaction–diffusion systems: the AC equation for the phase separation with interfacial motion and nonlinear reactions,  
and the KS chemotaxis model for the aggregation-driven blow-up dynamics~\cite{keller1970initiation, patlak1953random}. 
While the AC equation serves to demonstrate the generality of our method, the KS model provides a more challenging test problem and is the primary focus of the present study. In particular, we consider the two-dimensional KS chemotaxis model as the theme problem, which takes the form 
\begin{equation}\label{equ:KS}
\begin{aligned}
\partial_t u - \nabla \cdot ( \nabla u - \chi(u) \nabla v) &= f_u(u,v), & \bx\in \mathbb{T}^2 , t\in(0,T), \\
\partial_t v - \Delta v &= f_v(u,v), & \bx\in \mathbb{T}^2 , t\in(0,T),
\end{aligned}
\end{equation}
where $u$ and $v$ represent the cell density and chemical concentration, respectively. $\chi$ represents the chemotactic sensitivity, and $f_u$ and $f_v$ represent the production or decay rate.  For simplicity, we focus on periodic boundary conditions or unbounded domains, though the framework can be extended to more general boundary settings in future research.  
In two dimensions, it is well understood that both parabolic–elliptic and fully parabolic cases can exhibit finite-time blow-up to a Dirac-delta profile when the initial mass is sufficiently large~\cite{herrero1997blow, perthame2007transport}. Furthermore, in the three-dimensional parabolic–elliptic case, an alternative self-similar blow-up profile of the form $ C(T-t+\vert \mb x\vert ^2)^{-1}$ has also been identified in the literature~\cite{giga2011asymptotic,souplet2019blow}. 

On the numerical side, many notable methods have been developed for the KS systems. The finite-volume method~\cite{filbet2006finite,chertock2008second,chertock2018high} achieves efficient and positivity-preserving simulations for chemotaxis and haptotaxis models.  
Shen et al.~\cite{shen2020unconditionally} introduced an energy-dissipative and bound-preserving formulation, subsequently extended by Chen et al.~\cite{chen2022error} to a fully discrete finite element scheme for the 2D parabolic–elliptic KS, which reproduces finite-time blow-up under assumptions analogous to the continuous setting. 
Liu and Wang~\cite{liu2018positivity} reformulated the equations via the Le Ch\^atelier principle to construct positivity-preserving schemes. Discontinuous Galerkin methods have also been applied to KS systems~\cite{zhong2024direct, li2017local,epshteyn2009new}, in addition to various mesh-based finite element, finite volume, and domain-decomposition techniques for 3D fully parabolic models~\cite{epshteyn2019efficient, strehl2013positivity}. For particle-based approaches, the most relevant work is the recently proposed stochastic interacting particle–field algorithm~\cite{wang2025novel,hu2024stochastic,chertock2025hybrid}. However, this approach does not deal with the non-conservative equations at the particle level, limiting its applicability to the broader class of nonlinear non-conservative PDEs. 

We address this challenge by explicitly coupling a SDE-based advection--diffusion process with a particle-level birth--death mechanism for reactions, which overcomes the limitations of traditional methods and yields a particle measure discretization that remains robust near concentration and blow-up. We validate the method on two-dimensional KS benchmarks, including near blow-up and nonlinear growth, and compare with the standard grid-based method, observing the expected $1/\sqrt{N}$ scaling with the particle number.
The remainder of the paper is organized as follows: Section 2  and 3 present the stochastic branching method for a non-conservative ADR equation and the convergence analysis; Section 4 extends the approach for the KS systems; Section 5 reports numerical results and Section 6 discusses extensions and outlook.

\section{Stochastic branching method}
Let us start with the following scalar non-conservative advection-diffusion–reaction system on the $d$-dimensional torus $\mathbb{T}^d$ for a time $t \in (0, T]$:
\begin{subequations}\label{equ:reaction-diffusion}
\begin{align}
\partial_t u &= \nabla \!\cdot\! \big( a(u) u \big) + D \Delta u + r(u) u, & \text{for } (x, t) \in \mathbb{T}^d \times (0, T], \label{equ:pde} \\
u(x, 0) &= u_0(x), & \text{for } x \in \mathbb{T}^d. \label{equ:ic}
\end{align}
\end{subequations}
Here, $u(x,t)$ is the scalar field, $a(u): \mathbb{R}_{+} \to \mathbb{R}^d$ is the advective velocity, $D > 0$ is the diffusion coefficient, and $r(u): \mathbb{R}_{+} \to \mathbb{R}$ is the reaction rate. The system is initialized with $u_0(x)$. The extension to the full KS system~\eqref{equ:KS} is presented in Section~\ref{sec:extension}.

Let $S(N,t)=\{\bX^N_i\}_{i\in\mathcal I(N,t)}$ denote a set of particles at time $t$ whose initial population is $N$, where the particles' indices are numbered consecutively in the index set $\mathcal I(N,t)$ and $\bX^N_i(t)$ denote the position of the $i\mhyphen$th particle. The associated empirical (finite) measure is
\begin{equation}
    t \to \mu^N(t,\intd \bx) =  \frac{1}{N}\sum_{i\in \mathcal I(N,t)}\delta_{\bX^N_i(t)},
\end{equation}
where $\delta_\bx $ denotes the Dirac measure at $\bx\in \mathbb{T}^d$. The core idea of particle-based methods is to approximate the PDE solution by the limit of the empirical measure
\[
\mu(t,\mathrm{d}\bx) = \lim_{N\to\infty} \mu^N(t,\mathrm{d}\bx).
\]
Since Eq. \eqref{equ:reaction-diffusion} is non-conservative, the total mass of the system is not preserved. Consequently, both $\mu^N$ and its limit $\mu$ should be viewed as finite measures rather than probability measures.

\subsection{Initial condition and Normalization}
To connect the solution $u(t,\bx)$ to the measure $\mu(t,\mathrm{d}\bx)=\rho(t,\bx) \intd \bx$, we introduce a linear transformation between them. 
Without loss of generality, we assume that the initial measure $\mu(0,\cdot)$ is a probability measure with total mass $1$, and define
\begin{equation*}
    u(t,\bx)\,\mathrm{d}\bx = Z \,\mu(t,\mathrm{d}\bx),
\end{equation*}
where 
\[
Z = \int_{\Omega} u(0,\bx)\,\mathrm{d}\bx
\]
denotes the total initial mass.
Thus $\rho(0,\cdot)$ is a probability density
\[
\rho(0,\mathbf{x}) \;=\; \frac{u(0,\mathbf{x})}{Z},
\]
and we can sample the initial particle positions independently from $\rho(0,\cdot)$ (e.g., via Metropolis–Hastings) to form $S(N,0)=\{\mathbf{X}^i_0\}_{i=1}^N$. 
Substituting $u=Z\rho$ into~\eqref{equ:reaction-diffusion}  yields
\begin{equation}\label{equ:normalized_diffusion_reaction}
\partial_t \rho 
= \nabla \cdot \left( \tilde a(\rho)\rho \right)+ D \Delta\rho 
  + \tilde r(\rho)\rho,
\end{equation}
with rescaled coefficients
\[
\tilde a(\rho) = a(Z\rho),  
\qquad \tilde r(\rho) = r(Z\rho).
\]
Since reactions generally change the total mass, $\int_{\Omega}\rho(t,\mathbf{x})\,\mathrm{d}\mathbf{x}$ may not equal $1$ for $t>0$. Rather, it tracks the mass of $u$ via $\int_{\Omega} u(t,\mathbf{x})\,\mathrm{d}\mathbf{x} = Z \int_{\Omega} \rho(t,\mathbf{x})\,\mathrm{d}\mathbf{x}$.

\subsection{Operator Splitting}
To approximate the solution of the PDE, we discretize the time interval $[0,T]$ with step size $\tau>0$ and grid points
\[
0 = t_0 < t_1 < \cdots < t_M = T,\qquad \tau=t_{n+1}-t_n.
\]
Then an operator-splitting strategy is used to separate the dynamics according to their physical mechanisms. 
Equation~\eqref{equ:normalized_diffusion_reaction} can be written as
\begin{equation}
\partial_t \rho = \mathcal{A}\rho + \mathcal{B}\rho ,
\end{equation}
where the operators $\mathcal{A}$ and $\mathcal{B}$ correspond to
\begin{align}
\mathcal{A}\rho &= \nabla \cdot \left( \tilde a(\rho)\rho \right) + D\Delta\rho  , \\
\mathcal{B}\rho &= \tilde r(\rho)\rho .
\end{align}
Here, $\mathcal{A}$ describes the advection--diffusion dynamics, while $\mathcal{B}$ accounts for the reaction mechanism (e.g., proliferation or death).  The evolution of the exact solution over $t$ can be expressed formally as
\begin{equation}
\rho(t,\bx) = \exp\left( t (\mathcal{A} + \mathcal{B}) \right)\rho(0,\bx).
\end{equation}
In practice, the operator $\exp\!\big(\tau(\mathcal{A}+\mathcal{B})\big)$ cannot be evaluated explicitly, 
we adopt the first-order Lie-Trotter splitting \cite{trotter1959product}, which is justified by the Trotter product formula:
\begin{equation}
\exp\!\left( t (\mathcal{A} + \mathcal{B}) \right) 
= \lim_{M\to\infty} \Big( \exp\!\left(\tfrac{t}{M} \mathcal{B}\right) 
                          \exp\!\left(\tfrac{t}{M} \mathcal{A}\right) \Big)^M .
\end{equation}
At the discrete level, we approximate $\rho(t_{n+1},\cdot)$ by
\begin{equation}
\rho_{n+1} = \exp(\tau \mathcal{B}) \exp(\tau \mathcal{A}) \rho_n ,
\end{equation}
for $n=0,\cdots,M-1$.
This scheme provides a simple and efficient first-order approximation; 
higher-order splittings, such as Strang splitting \cite{strang1968construction}, may be employed for improved accuracy.
In practice, this decomposition allows us to design particle-based updates for each subprocess separately:
\begin{equation}
\rho_n^\ast = \exp(\tau \mathcal{A})\rho_n, 
\qquad 
\rho_{n+1} = \exp(\tau \mathcal{B})\rho_n^\ast.
\end{equation}

\subsection{Advection Diffusion Step}
\label{sec:adevection_diffusion}
In this step, we propagate particles under the advection–diffusion dynamics, where each particle follows a stochastic path determined by both a drift and a diffusion term. At time step $n$, let $S^N_n = \{\bX^i_{N,n}\}_{i\in \mathcal I_n^N}$ denote  a set of particles  whose initial population is $N$, which induce the empirical measure 
\[
\mu^N_{n}= \frac{1}{N}\sum_{i\in \mathcal I_n^N}\delta_{\bX_{N,n}^i}
\] 
approximates the solution at step $n$, i.e.,  $\mu^N_{n} (\intd \bx) \approx \mu_n(\intd \bx) = \rho_n(\bx)\intd \bx$. 
Our goal is to evolve these particles according to the transition probability $p(\bx,s;\by,t)$, which leads to an updated empirical measure
\[\mu_n^{N,*} = \int p(\bx,t_n;\by,t_n+\tau)  \mu_n^N(\intd \bx) =   \frac{1}{N}\sum_{i\in \mathcal I_n^N}\delta_{\bX_{N,n}^{i,*}}
\]  
reflecting the action of the advection--diffusion operator on the density.
By the Kolmogorov forward equation, the transition density $p(\bx,s;\by,t)$ satisfies 
\begin{equation}
    \partial_t p = \tilde {\mathcal A }p := - \nabla \cdot(b  p) + \frac{1}{2}\nabla^2:(\sigma \sigma^\top p) 
\end{equation}
with initial condition $\lim_{t\to s^+} p(\bx,s;\cdot,t) = \delta_\bx$. By matching the differential operator $\mathcal A$ with the target PDE \eqref{equ:reaction-diffusion}, we have 
\begin{equation}
    b(x) = -\tilde a(\rho^N_n(x)), \quad \frac{1}{2}\sigma \sigma^\top = D,
\end{equation}
where $\rho^N_n(x)$ represents the interpolated probability density of $\mu_n^N$ and will be discussed in Sec. \ref{sec:interpolation}.
Accordingly, the evolution of each particle $\bX^i_n$ as a stochastic process satisfying the It\^o stochastic differential equation
\begin{equation}\label{equ:overd_langevin}
    \intd \bX_{i,n}^N = b(\bX_{i,n}^N) \intd t  + \sigma \intd \bW^N_{i}(t),
\end{equation}
where $\bW^i_{N}(t)$ is a $2\mhyphen$dimensional Brownian motion. 
In practice, we approximate Eq. \eqref{equ:overd_langevin} using an Euler-Maruyama scheme with step size $\tau$, 
\begin{equation}
 \bX_{i,n}^{N,*} = \mathbf{X}_{i,n}^{N} + b(\bX_{N,n}^i) \tau  + \sigma \sqrt{\tau} \boldsymbol{\xi}^N_{i,n},
 \label{eq:euler_maruyama}
\end{equation}
where $\boldsymbol{\xi}^i_{N,n}$ represent identically independent distributed normal random variables; see Algorithm \ref{alg:sde} for the summary of the particle propagation for the advection--diffusion step. 

\begin{algorithm}
\caption{SDE propagation of the advection--diffusion step}\label{alg:sde}
\begin{algorithmic}
\Require Particle set: $S=\{\mathbf{X}_i\}_{i\in\mathcal I}$; time step $\tau>0$
\Require Drift term $b(\cdot)$
\Require Diffusion strength $\sigma$
\Ensure Propagated set $S_*=\{\mathbf{X}^{i,*}\}_{i\in\mathcal I}$
\Procedure{SDE\_Propagate}{$S,\,\tau,\,b,\,\sigma$}
  \State $S_{*}\gets [\ ]$
  \For{$i\in\mathcal I$}
    \State Sample $\boldsymbol{\xi}_i \sim \mathcal N(0,I_d)$ independently
    \State $\mathbf{X}^*_i \gets \mathbf{X}_i \;+\; b(\mathbf{X}_i;\mathcal F)\,\tau \;+\; \sqrt{\tau}\,\sigma\,\boldsymbol{\xi}_i$
    \State Append $\mathbf{X}^{*}_i$ to $S_{*}$
  \EndFor
  \State \Return $S_{*}$
\EndProcedure
\end{algorithmic}
\end{algorithm}

\subsection{Branching Step}
\label{sec:branching}
This step modifies the mass of the empirical measure and change the particle number. Given the intermediate empirical measure from the advection--diffusion step $\mu_{N,n}^*$, we aim to evolve this measure under the action of the reaction operator $\mathcal  B$ over a time step $\tau$ following
\[
\partial_t \rho = c_n(\mathbf{x})\,\rho,
\qquad c_n(\mathbf{x}) := \tilde r\!\big(\rho_n^\ast(\mathbf{x})\big),
\]
such that the resulting empirical measure
\begin{equation}
    \mu^N_{n+1} = \frac{1}{N} \sum_{i\in \mathcal I (N,(n+1)\tau)} \delta_{\bX^N_{i,n+1}}
\end{equation}
approximates the updated density $\mu^N_{n+1}(\intd \bx ) = \rho^N_{n+1}(\bx) \intd \bx \approx {\rm e}^{c_n(\mathbf{x}) \tau} \rho_{n}^{N,*}(\bx) \intd \bx$. Therefore, the particles located at $\bx$ should birth and death with probability $c(\bx)=\tilde r(\rho^*(\bx))$ depending on the positiveness of $\tilde r$.
To implement this step, we adopt a branching process. For each particle $\bX^{N,\ast}_{i,n}$, we associate a branching rule determined by $c_i:=c_n(\mathbf X^{N,\ast}_{i,n})$. Specifically, over the interval $[t_n,t_{n+1}]$, each particle undergoes one of the following stochastic events:
\begin{itemize}
  \item \emph{If \(c_i\ge 0\) (birth):} keep the parent and draw \(K_i\sim \mathrm{Poisson}\!\big(e^{c_i\tau}-1\big)\); output \(1+K_i\) identical copies at \(\mathbf X^{N,\ast}_{i,n}\).
  \item \emph{If \(c_i<0\) (death):} keep the parent with probability \(p_i=e^{c_i\tau}\), otherwise remove it.
\end{itemize}

Let the resulting set be \(S_{n+1}^N=\{\mathbf X^N_{i,n+1}\}_{i\in\mathcal I_{n+1}^N}\) and
\[
\mu_{n+1}^{N}(\mathrm{d}\mathbf{x})
\;=\;
\frac{1}{N}\sum_{i\in\mathcal I_{n+1}^N}\delta_{\mathbf X^N_{i,n+1}}(\mathrm{d}\mathbf{x}).
\]

\begin{proposition}
    Let $S=\{\bX_i\}_{i\in \mathcal I}$ be a set of particles, whose empirical measure is defined as $\mu (\intd \bx) = \frac{1}{N}\sum_{i\in \mathcal I }\delta_{\bX_i}$. If the particles are updated according to the above rules then the updated particle set $S_+=\{ \bX_{i,+}\}_{i\in {\mathcal I}_+}$ and corresponding  measure $ \mu_+ (\intd \bx) = \frac{1}{N}\sum_{i\in  {\mathcal  I}_+}\delta_{ \bX_{i,+}}$ should satisfies 
    \begin{equation}
        \mathbb E [\left<f, \mu_+\right>] =   \left<f,\exp( c\tau)\mu\right>,
    \end{equation}
    for all $f\in C^2_b(\mathbb R)$.
\label{prop:stochastic_branch}
\end{proposition}

\begin{proof} Let $\lambda_i = {\rm e}^{c(\mb X^i) \tau} -1 $ denote the  rate of the Poisson process for the birth branch with $c(\mb x^i) \ge 0$. For any test function $f\in C^2_b(\mathbb R^+, \mathbb R^2)$, we have
\begin{equation}
\begin{aligned}
     & \mathbb E\left[\left<f,  \mu_+\right>\right]
    \\=& \frac{1}{N} \sum_{i\in \mathcal I ,c(\bX_i)>0} \sum_{j=0}^{\infty} 
    \left[(j+1) \frac{e^{-\lambda_i} \lambda_i^j}{j!}\right] f(\bX_i)
     + \frac{1}{N} \sum_{i\in \mathcal I , c(\bX_i)<0}
    \exp( c(\bX_i)\tau)f(\bX_i) \\
    =& \frac{1}{N}\sum_{i\in \mathcal I ,c(\bX_i)>0} 
    \left(
    \exp(c(\bX_i)\tau)f(\bX_i)
    \right) 
     + \frac{1}{N}\sum_{i\in \mathcal I , c(\bX_i)<0}\left(
    \exp( c(\bX_i)\tau)f(\bX_i) \right)\\
    = & \frac{1}{N}\sum_{i\in \mathcal I}\exp( c(\bX_i)\tau)f(\bX_i) 
    = \left<f,\exp( c\tau)\mu\right>,
\end{aligned}
\end{equation}
where the second equality follows from the fact that the expectation of a Poisson process equal to the rate.
\end{proof}
By choosing $c(\mb x) = \tilde{r}(\rho^{\ast}(\mb x))$, this stochastic branching process realizes the operator $\exp(\mathcal{B}\tau)$ at the particle level. We summarize the process in Algorithm \ref{alg:birthdeath}.

\begin{algorithm}
\caption{Birth--Death Step}\label{alg:birthdeath}
\begin{algorithmic}
\Require Propagated particles $S=\{\mathbf X_{i}\}_{i\in\mathcal I}$, time step $\tau>0$, rate function $r:\mathbb R^2\to\mathbb R$
\Ensure Post-branching set $S_{+}=\{\mathbf X_{i}\}_{i\in\mathcal I_{+}}$
\Procedure{Birth\_Death}{$S,\,\tau,\,r(\cdot)$}
    \State $S_{n+1}\gets [\ ]$
    \For{$i\in\mathcal I$}
        \State $r_i \gets r(\mathbf X_{i})$
        \If{$r_i>0$} \Comment{birth}
            \State Append $\mathbf X_{i}$ to $S_{+}$ \Comment{keep parent}
            \State $p_{\mathrm{birth}} \gets \exp(r_i\,\tau) - 1$
            \State Draw $K_i \sim \mathrm{Poisson}\!\big(p_{\mathrm{birth}}\big)$, append $K_i$ duplicates of $\mathbf X_{i}$ to $S_{+}$
        \ElsIf{$r_i<0$} \Comment{death}
            \State $p_{\mathrm{survive}}\gets \exp(r_i\,\tau)$ 
            \State With probability $p_{\mathrm{survive}}$, append $\mathbf X_{i}$ to $S_{+}$
            \State \quad (otherwise the particle is removed)
        \Else \Comment{$r_i=0$}
            \State Append $\mathbf X_{i}$ to $S_{+}$
        \EndIf
    \EndFor
    \State \Return $S_{+}$
\EndProcedure
\end{algorithmic}
\end{algorithm}

\subsection{Nonlinear Interpolation}
\label{sec:interpolation}
To evaluate the functions $\tilde a(\rho)$ , and $\tilde r(\rho)$ , that govern the drift and reaction dynamics in our particle-based scheme, we need to reconstruct the continuous density 
$\rho$ from its empirical measure. Specifically, given an empirical measure
\[
\mu^N_n= \frac{1}{N}\sum_{i\in \mathcal I^N_n}\delta_{\bX_{i,n}^N},
\] 
we seek a smooth approximation of $\rho^N_n$ that allows pointwise evaluation and differentiation as required by the operator $\mathcal{A}$ and $\mathcal{B}$.

In this work, we choose to interpolate the empirical distribution in the Fourier space. Specifically, we project the density onto a finite set of Fourier basis functions, allowing efficient representation and differentiation. 

The Fourier coefficients are evaluated from the particle locations $\bX_{N,n}^i$ under periodic boundary conditions. Without loss of generality, we take the periodic domain to be $\Omega=[-\pi,\pi]^d$ , from which the reconstructed density $\rho_n(t,x)$. To be more specific, we define a Fourier interpolation operator $\mathcal P_K$
that maps the empirical particle distribution into a truncated Fourier representation 
\[
\rho^N_n(\bx) = \mathcal{P}_K(\mu^N_n)(\bx) 
          = \sum_{\vert k_1\vert ,\vert k_2\vert \leq K} \hat{\rho}_n(k_1,k_2,\cdots,k_d)\,\phi_{k_1}(\bx_1)\phi_{k_2}(\bx_2)\cdots\phi_{k_d}(\bx_d),
\]
where the basis functions are chosen as
\[
\phi_k(x) =
\begin{cases}
\frac{1}{\sqrt{2\pi}}, & k=0, \\
\frac{1}{\sqrt{2\pi}} \cos(k  x), & k \in \mathbb{Z}_+, \\
\frac{1}{\sqrt{2\pi}}\sin(k  x), & k \in \mathbb{Z}_- .
\end{cases}
\]

The Fourier coefficients are computed directly from the empirical measure:
\begin{equation}
\begin{aligned}
    \hat{\rho}_n(k_1,k_2,\cdots,k_d) &= \int \phi_{k_1}(y_1)\phi_{k_2}(y_2)\cdots \phi_{k_d}(y_d)\,\mu^N_n(\intd\by)
\end{aligned}
\label{eq:Fourier_mode}
\end{equation}
The process is shown in Algorithm \ref{alg:interpolate}. 
This choice is motivated by the simplicity and spectral accuracy of Fourier methods for smooth periodic functions. However, we note that other functional spaces, such as polynomial bases, wavelets, or kernel-based representations, could offer better adaptability in non-periodic or highly localized settings. We leave the exploration of such alternatives to future work.

\begin{algorithm}
\caption{Interpolation Step (fixed-$N$ normalization)}\label{alg:interpolate}
\begin{algorithmic}
\Require Particle set: $S=\{ \mathbf{X}^i \}_{i\in \mathcal{I}}$; fixed initial population $N$; truncation index set $\mathcal{M}_K$; domain descriptor $\Omega$; basis $\{\psi_m\}_{m\in\mathcal{M}_K}$ with (optionally) known gradients $\{\nabla\psi_m\}_{m\in\mathcal{M}_K}$
\Ensure Truncated field $\tilde{\rho}_n(\mathbf{x})$ (and, if needed, $\nabla\tilde{\rho}_n(\mathbf{x})$)
\Procedure{Interpolate}{$S_n,\,N,\,\mathcal{M}_K,\,\Omega,\,\{\psi_m\}$}
    \State \textbf{(Coefficient projection with fixed-$N$ normalization)}
    \State For each $m\in\mathcal{M}_K$, set
    \[
      a_m \;\gets\; \frac{1}{N}\sum_{i\in\mathcal{I}} \psi_m\!\big(\mathbf{X}^i\big).
    \]
    \State \textbf{(Series assembly)}
    \[
      \tilde{\rho}_n(\mathbf{x}) \;\gets\; \sum_{m\in\mathcal{M}_K} a_m\,\psi_m(\mathbf{x}).
    \]
    \State \textbf{(Optional gradients for drift/evaluation)}
    \If{gradients are required and $\{\nabla\psi_m\}$ are available}
      \[
        \nabla\tilde{\rho}_n(\mathbf{x}) \;\gets\; \sum_{m\in\mathcal{M}_K} a_m\,\nabla\psi_m(\mathbf{x}).
      \]
    \EndIf
    \State \Return $\tilde{\rho}_n$ (and $\nabla\tilde{\rho}_n$ if computed)
\EndProcedure
\end{algorithmic}
\end{algorithm}

\section{Error Estimations}
In this section, we analyze the error between the numerical and exact solutions for Equation (\ref{equ:reaction-diffusion}). First, we introduce the projected PDE 
\begin{equation}\label{equ:projected}
    \partial_t u_K  = \nabla \cdot \left( a (\mathcal P_K u_K) u_K\right) + D \Delta u_K  + r\left(\mathcal P_K u_K\right)  u_K, 
\end{equation}
where $\mathcal P_K$ means the projection operator on the Fourier space up to the smallest $K$ modes. To analyze the convergence of the stochastic branching algorithm, we introduce the following empirical measure associated with algorithm,
Let $S_K(N,t) = \{\bX^N_{i,K}(t) \}_{i\in \mathcal I_K(N,t)}$ denote the set of particles at time $t$, where
 $\mathcal{I}_K(N,t)$ is the index set. We define the following empirical measure \[
 \mu^N(t,\intd x) = \frac{1}{N} \sum_{i\in \mathcal I(N,t)} \delta_{\bX^N_{i,K}(t)}(\intd \bx),
 \] which represents the random, finite measure obtained from the particle configuration at time $t$. In the weak form, for every test function $f \in C^2(\mathbb T^d)$,
 \begin{equation}
     \begin{aligned}
\langle f, \mu^N(t)\rangle
&= \langle f, \mu^N(0)\rangle
 - \int_0^t
   \!\!\langle\nabla f,\,a(P_K\mu^N)\mu^N\rangle\,\intd\tau
 + \int_0^t
   \!\!\langle\Delta f,\,D\,\mu^N\rangle\,\intd\tau \\
&\quad + \int_0^t \langle f, r(\mathcal P_K \mu^N)\mu^N\rangle\,\intd\tau 
 + \frac{1}{N}\int_0^t\sum_{i\in\mathcal I(N,\tau)}
   \sqrt{2D}\,\nabla f(\bX_{i,K}^N(\tau))\!\cdot\! \intd \bW^N_i(\tau)\\
 &\quad + \frac{1}{N}\int_0^t\!\sum_{i\in\mathcal I(N,\tau),r(\mathcal P_K \mu^N(\bX_{i,K}))>0}\!
   f(\bX_{i,K}^N(\tau))\, \intd \mathcal R^{N,\text{birth}}_i(\tau)\\
&\quad- \frac{1}{N}\int_0^t\!\sum_{i\in\mathcal I(N,\tau),r(\mathcal P_K \mu^N(\bX_{i,K}))<0}\!
   f(\bX_{i,K}^N(\tau))\, \intd \mathcal R^{N,\text{death}}_i(\tau),
\end{aligned}
 \end{equation}
where two jumping processes are added to the measure for $\sigma\in \mathbb R^+$ 
\[
\begin{aligned}
   \mathcal R^{N,birth}_{i^*}( \sigma) = Q^{N,birth}_i \left(\int_{0}^{ \sigma} \mathbf 1 _{S(N,\tau)}(i) r(\mathcal P_K \mu^N)\intd \tau\right),\\
    \mathcal R^{N,death}_{i^*}( \sigma) = Q^{N,death}_i \left(\int_{0}^{\sigma} -\mathbf 1 _{S(N,\tau)}(i) r(\mathcal P_K \mu^N)\intd \tau\right), 
\end{aligned}
\]
 where $Q^{N,birth}_i$, $Q^{N,death}_i$ are independent standard Poisson processes. 
 The following theorem gives the error estimation of our algorithm with the projected PDE.
\begin{proposition}
\label{prop:u_K_mu}
Let $\mathbb{T}^d = (\mathbb{R}/2\pi\mathbb{Z})^d$, and assume that the coefficients $a$ and $r$ are bounded and lipschitz continuous with a uniform constant $C_{b,lip}$, i.e. for any $x, y\in \mathbb R$, we have
    \[
    |a(x)|,|r(x)| < C_{b,lip} \quad |a(x)-a(y)|,|r(x)-r(y)| < C_{b,lip} |x-y|.
    \]
    Suppose that for a finite time horizon $[0,T)$: The number of particles in the stochastic particle algorithm is uniformly bounded $|\mathcal{I}(N,t)| \le C_{N}N$; The projected PDE admits a unique bounded solution $u_K$, with $\|u_K(t,\cdot)\|_{L^\infty}\leq C_u$ for all $t\in [0,T]$.
    We define the $H^{-s}$ error
    $$E_{K}(t) = \frac{1}{2} ||u_K(t) - \mu^N(t)||_{H^{-s}}^2$$
    Assume $s > d/2 + 1$. Then, we have:
    \[
   \frac{d}{dt} \mathbb E  E_K(t) \leq 
C(C_N,C_u,C_{b,lip},N,D,s,d)  \mathbb E E_K(t) 
    \]
    and therefore \[
    \mathbb E E_K(t) \leq \exp\left(C(C_N,C_u,C_{b,lip},N,D,s,d)t\right)\mathbb E E_K(0).
    \]  
\end{proposition}    
\begin{proof}
    Let us denote the error as $\hat{e}_{K} = u_{K}-\mu^{N}$ and define the error functional $E_{K}(t) = \frac{1}{2}||\hat{e}_{K}||_{H^{-s}}^{2} = \frac{1}{2}\langle\hat{e}_{K},(I-\Delta)^{-s}\hat{e}_{K}\rangle$. We also introduce the auxiliary function $\phi_{K} = (I-\Delta)^{-s}\hat{e}_{K}$. A key identity from these definitions is:
    \begin{equation}
        \begin{aligned}
            ||\phi_{K}||_{H^{s}}^{2} =& \langle \phi_K, (I-\Delta)^s \phi_K \rangle = \langle (I-\Delta)^{-s} \hat{e}_K, (I-\Delta)^s (I-\Delta)^{-s} \hat{e}_K \rangle\\
            =& \langle (I-\Delta)^{-s} \hat{e}_K, \hat{e}_K \rangle = ||\hat{e}_K||_{H^{-s}}^2= 2E_{K}(t)
        \end{aligned}
    \end{equation}
We apply It\^o's formula to the stochastic process $E_K(t)$. The dynamics of $u_K$ are deterministic, while the dynamics of $\mu^N$ are stochastic. The time derivative of the expectation $\mathbb{E}E_K(t)$ is given by the drift terms plus the quadratic variation terms from the martingale parts.
\begin{equation}
 \begin{aligned}
 \frac{d}{dt} \mathbb E  E_K(t)  = &   \quad \mathbb E \left[\left\langle \nabla  \phi_K , a(\mathcal P_K \mu^N) \mu^N-  a(\mathcal P_K u_K) u_K   \right\rangle \right] \\
  & +  \mathbb E \left[\left\langle \Delta  \phi_K , D (u_K - \mu^N) \right\rangle \right] \\
  & + \mathbb E \left[\left\langle   \phi_K , r(\mathcal P_K u_K)u_K - r(\mathcal P_K \mu^N)\mu^N \right\rangle \right]  \\
  & +  \mathbb E \left[\frac{D}{N^2} \sum_{i\in \mathcal I (N,t)}  \left|\nabla   \phi_K(\bX_N^i(t) )\right|^2 \right] \\
  & + \mathbb E \left[\frac{1}{2N^2} \sum_{i\in \mathcal I (N,t)} \left(  \phi_K(\bX_N^i(t))\right)^2 \left|r(\mathcal{P}_K \mu^N(\bX_N^i(t)))\right|\right]\\
  & = \mathbb E \left[I_1\right]+\mathbb E \left[I_2\right]+\mathbb E \left[I_3\right]+\mathbb E \left[I_4\right]+
  \mathbb E \left[I_5\right].
 \end{aligned}
\end{equation}
\textbf{Term $I_1$ (Advection)}: We split this term: 
\begin{equation}
    \begin{aligned}
        I_1 
        = &  
         \left\langle \nabla  \phi_K,  a(\mathcal P_K \mu^N ) ( \mu^N - u_K )\right\rangle 
        +
        \left\langle \nabla  \phi_K,  \left( a(\mathcal P_K \mu^N) - a(\mathcal P_K u_K) \right) u_K \right\rangle \\
        =& I_{11} + I_{12}.
    \end{aligned}
\end{equation}
For $I_{11}$, using duality $\langle H^{s-1}, H^{-s+1} \rangle$, boundedness of $a$, and Young's inequality:
\[
\begin{aligned}
    I_{11}
        \leq & 
        \left|\left\langle a(\mathcal P_K \mu^N)  \nabla  \phi_K,   (u_K - \mu^N )\right\rangle\right|\\
        \leq& \left\|a(\mathcal P_K \mu^N)  \nabla  \phi_K\right\|_{H^{s-1}} \left\|(u_K - \mu^N )\right\|_{H^{-s+1}} \\
        \leq& C_{b,lip} \left\|\nabla  \phi_K\right\|_{H^{s-1}} \left\|(u_K - \mu^N )\right\|_{H^{-s+1}} \\
        \leq& \frac{C_{b,lip}}{2\epsilon}\left\|\nabla  \phi_K\right\|_{H^{s-1}}^2 + \frac{C_{b,lip}\epsilon}{2}\|\hat{e}_K\|^2_{H^{-s+1}}
        \\
         \leq& \frac{C_{b,lip}}{2\epsilon}\left\|  \phi_K\right\|_{H^{s}}^2 + \frac{C_{b,lip}\epsilon}{2}\|\hat{e}_K\|^2_{H^{-s+1}}\\
        = & \frac{C_{b,lip}}{\epsilon} E_K(t)+ \frac{C_{b,lip}\epsilon}{2}\|\hat{e}_K\|^2_{H^{-s+1}}.
\end{aligned}
\]
For $I_{12}$, using Lipschitz continuity of $a$ and boundedness of $u_K$:
\[
\begin{aligned}
     I_{12} \leq & \left\|u_K \nabla \phi_K\right\|_{H^{s-1}} \| a(\mathcal  P_K \mu^N) - a(\mathcal P_K u_K)\|_{H^{-s+1}}\\
     \leq&  C_{b,lip} C_u\left\|\nabla  \phi\right\|_{H^{s-1}} \left\|\mathcal P_K\hat{e}_K\right\|_{H^{-s+1}} \\
     \leq & C_{b,lip} C_u\left\| \phi_K\right\|_{H^{s}} \left\|\hat{e}_K \right\|_{H^{-s+1}}
     \\ 
     \leq& \frac{C_{b,lip} C_u}{2\epsilon}E_K(t)
        + \frac{C_{b,lip} C_u\epsilon}{2}\|\hat{e}_K\|^2_{H^{-s+1}}.
\end{aligned}
\]
In total, we have
\[
I_1\leq \frac{C_{1}}{\epsilon}E_K(t)
        + \frac{C_{1}\epsilon}{2}\|\hat{e}_K\|^2_{H^{-s+1}}.
\]
\textbf{Term $I_2$ (Diffusion): }
\[
\begin{aligned}
I_{2} = & D\left\langle (I-(I-\Delta))  \phi ,  \hat{e}_K \right\rangle\\
   =& D\left\langle  \phi_K ,  \hat{e}_K \right\rangle -  D\left\langle (I-\Delta)  \phi_K ,  \hat{e}_K \right\rangle\\
   = & 2D E_K(t) - D\left\langle (I-\Delta)^{-s+1}  \hat{e}_K ,  \hat{e}_K \right\rangle\\
   = & 2D E_K(t) - D\|\hat{e}_K\|_{H^{-s+1}}^2.
\end{aligned}
\]
\textbf{Term $I_3$ (Reaction):} This follows the same logic as $I_1$.
\[
\begin{aligned}
I_3 
=& \left\langle   \phi_K , \left(r\left(\mathcal P_K u_K\right)-r\left(\mathcal P_K \mu^N\right)\right)u_K\right\rangle 
+ \left\langle   \phi_K , r(\mathcal P_K \mu^N)\left(u_K-\mu^N\right) \right\rangle \\
\leq &  C_{b,lip} \|u_K\|_{L^\infty}E_K(t) +C_{b,lip} E_K(t)
\leq  C_3 E_K(t).
\end{aligned}
\]
\textbf{Term $I_4$ (Diffusion QV):}
Using the bound $|\mathcal{I}(N,t)| \le C_N N$ and Sobolev embedding:
\begin{equation}
    \begin{aligned}
    I_4  \leq &
    \frac{D C_N}{N} \|\nabla \phi_K\|^2_{L^\infty}
    \leq \frac{D C_N}{N}\| \phi_K\|^2_{H^{s}}= \frac{C_4}{N} E_K(t).
\end{aligned}
\end{equation}
\textbf{Term $I_5$ (Branching QV):}
Using $s > d/2+1 > d/2$ and $||r||_{L^\infty} \le C_{b,lip}$:
\[
\begin{aligned}
    I_5 
    \leq& \frac{C_N}{2N}\|r\|_{L^\infty}\|\phi_K\|^2_{L^\infty}
    \leq  \frac{C_NC_{b,lip}}{2N}\|\phi_K\|^2_{H^{s}}
          = \frac{C_5}{N} E_K(t).
\end{aligned}
\]
\textbf{Conclusion:}
Combining all the terms, we get:

$$\frac{d}{dt}\mathbb{E}E_{K}(t) \le \left( \frac{C_1}{\epsilon} + 2D + C_3 + \frac{C_4}{N} + \frac{C_5}{N} \right) \mathbb{E}E_K(t) + \left( \frac{C_1 \epsilon}{2} - D \right) \mathbb{E}||\hat{e}_{K}||_{H^{-s+1}}^{2}.$$

Since $D > 0$, we can choose $\epsilon$ sufficiently small, for example $\epsilon = D/C_1$, such that the last term is negative, which yields 
\[
\begin{aligned}
    \frac{d}{dt} \mathbb E  E_K(t) \leq 
C(C_N,C_u,C_{b,lip},N,D,s,d) \mathbb E  E_K(t). 
\end{aligned}
\]
By Gr\"onwall,
\[
\mathbb E E_K(t) \leq \exp(C(C_N,C_u,C_{b,lip},N,D,s,d)t)\mathbb E E_K(0).
\]
\end{proof}

Next, we estimate the error between the projected PDE \eqref{equ:projected} and the original PDE \eqref{equ:reaction-diffusion}.
\begin{proposition}
\label{prop:u_u_K}
Let $\mathbb{T}^d = (\mathbb{R}/2\pi\mathbb{Z})^d$, and assume that the coefficients $a$ and $r$ are bounded and lipschitz continuous with a uniform constant $C_{b,lip}$, i.e. for any $x, y\in \mathbb R$, we have
    \[
    |a(x)|,|r(x)| < C_{b,lip} \quad |a(x)-a(y)|,|r(x)-r(y)| < C_{b,lip} |x-y|.
    \]
    Suppose that for a finite time horizon $[0,T)$: $u$ is a sufficiently smooth solution of the nonlinear PDE~\eqref{equ:reaction-diffusion} on $\Omega$ with periodic  boundary conditions, and $u_K$ is the solution of the projected PDE~\eqref{equ:projected}. 
If $u(\cdot,t)\in \mathbb H^m(\Omega)$ for some $m>1$, then the error $e_K=u- u_K$ satisfies
\[
\|e_K(t)\|_{L^2}^2 \;\le\; C_T\Big(\|e(0)\|_{L^2}^2 + K^{-2m}\|u\|_{L^2(0,t;\mathbb H^m)}^2\Big),
\]
for some constant $C_T>0$ depending only on $T$ and the coefficients.
\end{proposition}

\begin{proof}
    Let $\eta(t) = \frac{1}{2}\|e_K(t,\cdot)\|^2_{L^2(\Omega)}  = \frac{1}{2}\|u(t,\cdot)- u_K(t,\cdot)\|^2_{L^2(\Omega)} $, then we have 
    \begin{equation}
        \begin{aligned}
    \frac{\partial}{\partial t} \eta(t)
        = & -\int_\Omega \nabla (u-u_K)\cdot (a(u)u - a(\mathcal P_K u_K )u_K )\intd \bx \\
        & + D\int_\Omega (u-u_K)  \Delta (u -  u_K  )\intd \bx \\
        & + \int_\Omega (u-u_K)  (r(u)u - r(\mathcal P_k  u_K ) u_K) \intd \bx  \\
        = &  E_1 + E_2 + E_3
        \end{aligned}
    \end{equation}
Notice that for $I_1$, we have
\[
\begin{aligned}
    E_1 =& -\int_\Omega \nabla(u-u_K) \cdot \left(\left(a(u) - a(\mathcal P_K u_K )\right)u\right) \intd \bx\\
    &- \int_\Omega \nabla(u-u_K) \cdot \left(a(\mathcal P_K u_K )(u  - u_K )\right)\intd \bx = E_{11}+ E_{12}.\\
\end{aligned}
\]
For $E_{11}$, we have 
\[
\begin{aligned}
E_{11} \leq & \|\nabla e_K \|_{L^2} \|\left(a(u) - a(\mathcal P_K u_K )\right) u \|_{L^2}\\
 \leq& \|\nabla e_K \|_{L^2} C_{b,lip}\left\|u\|_{L^\infty}\right\|u - \mathcal P_K u_K \|_{L^2}\\
\leq&  \frac{\epsilon C_{b,lip} \left\|u\right\|_{L^\infty} }{2} \|\nabla e_K \|_{L^2} 
+\frac{C_{b,lip}\|u\|_{L^\infty}}{2\epsilon} \|u - \mathcal P_K u_K\|_{L^2}^2.
\end{aligned}
\]
Assume that $u\in \mathbb H^m$
\[
\begin{aligned}
\|u - \mathcal P_K u_K\|_{L^2} \leq& \|u - \mathcal P_K u\|_{L^2} + \|\mathcal P_K u - \mathcal P_K u_K\|_{L^2}\\
=& \|(I - \mathcal P_K) u\|_{L^2} + \|\mathcal  P_K (u - u_K)\|_{L^2} \\
\leq &\|(I - \mathcal P_K) u\|_{L^2} + \|e_K\|_{L^2}\\
\leq & C K^{-2m} \|u\|_{H^m}^2 + \|e_K\|_{L^2}.
\end{aligned}
\]
For $E_{12}$, we have 
\[
\begin{aligned}
    E_{12} \leq& \left| \int_\Omega \nabla(u-u_K) \cdot \left(a(\mathcal P_K u_K )(u  - u_K )\right)\intd \bx \right|\\
    \leq& \|\nabla e_K\|_{L^2}\left\|a(\mathcal P_K u_K )(u  - u_K ) \right\|_{L^2} \\
    \leq& C_{b,lip} \|\nabla e_K\|_{L^2}\left\|e_K \right\|_{L^2} \\
    \leq& \frac{\epsilon C_{b,lip}}{2}\|\nabla e_K\|_{L^2} +\frac{C_{b,lip}}{2\epsilon}\|e_K\|_{L^2}.
\end{aligned}
\]
Add two term together we have 
\[
E_1 \leq \frac{\epsilon C_1}{2}\|\nabla e_K\|_{L^2} +\frac{C_1}{2\epsilon}\| e_K\|_{L^2} + C_1 K^{-2m}\|u\|^2{\mathbb H^m},
\]
where $C_1$  depends on $ \|a\|_{L^\infty},\left\|a'\right\|_{L^\infty},\|u\|_{L^\infty}$.
For $E_2$, we have 
\[
E_2 \leq - D\int_\Omega (\nabla e_K )^2\intd \bx \leq -D\|\nabla e_K\|^2_{L^2}.
\]
For $E_3$, we have 
\[
\begin{aligned}
    E_3 = &   \int_\Omega (u-u_K)  (r(u) - r(\mathcal P_K  u_K )) u \intd \bx 
     +  \int_\Omega (u-u_K)  r(\mathcal P_K  u_K ) (u-u_K)  ) \intd \bx \\
    \leq & \|e_K\|_{L^2} C_{b,lip}\|u-\mathcal P_K u_K\|_{L^2 }
     +C_{b,lip} \|e_K\|_{L^2}^2\\
    \leq & \frac{1}{2}C_{b,lip} \|e_K\|_{L^2} ^2 + \frac{1}{2}C_{b,lip}( C K^{-2m} \|u\|_{H^m}^2 + \|e_K\|_{L^2})
     +C_{b,lip} \|e_K\|_{L^2}^2\\
    = & 2C_{b,lip}\|e_K\|_{L^2}^2 + \frac{ C K^{-2m}}{2}C_{b,lip}\|u\|_{H^m}^2\\
    = & C_3 \|e_K\|_{L^2}^2+ C_3 K^{-2m} \|u\|_{H^m}^2.
\end{aligned}
\]
where $C_3$ depends on $\|r'\|_{L^\infty}$ and $\|r\|_{L^\infty}$.

By choose $\epsilon$ small enough compared with $D$, we can have
\[
\frac{\partial }{\partial t }\eta \leq(C_1+C_2)\eta(t) + (C_1+C_3) K^{-2m} \|u\|_{H^m}^2.
\]
Applying the Gr\"onwall inequility yields
$$\eta(t) \le \eta(0)e^{(C_1+C_2)t} + \int_0^t e^{(C_1+C_2)(t-s)} (C_1+C_3) K^{-2m} \|u(s)\|_{H^m}^2 \intd s.$$
Therefore, we have
$$\|e_K(t)\|_{L^2}^2 \le C_T \Big( \|e_K(0)\|_{L^2}^2 + K^{-2m} \|u\|_{L^2(0,t;H^m)}^2 \Big).$$
\end{proof}

By Proposition \ref{prop:u_K_mu} and \ref{prop:u_u_K}, if we choose $u_K = u$ at $t=0$ and the initial measure $\mu^N$ is an unbiased sampling of $u_0(x)$, i.e., $\mathbb{E} E_K(0) \le  C_s/N$, we have 
\begin{equation}
\mathbb{E}\Vert u(t) - \mu^N(t)\Vert_{H^{-s}}^2 \le C_T \left(C_s/N + K^{-2m} \|u\|_{L^2(0,t;H^m)}^2 \right).     
\end{equation}
The numerical details to implement this step is summarized in Algorithm \ref{alg:linear}.

\begin{algorithm}
    \caption{Stochastic Branching for linear scalar reaction--diffusion--advection equation} \label{alg:linear}
    \begin{algorithmic}
        \Require Initial condition of PDE $u_0(\bx)$, domain descriptor $\Omega$;
        \Require Fixed initial population size $N$
        \Require Time step $\tau>0$, horizon $T$; truncation $\mathcal M_K$; basis $\{\psi_m\}_{m\in\mathcal M_K}$
        \Require Driven term $b(\bx)$
        \Require Diffusion strength $\sigma(\bx)$
        \Require Reaction maps $r(\bx)$
        \Ensure  Solution set of PDE $\mathcal U = \{u_n\}_{n=0}^{T/\tau}$
        \Procedure{Stochastic\_Branching}{$u_0$,$N$,$\tau$,$\mathcal M_K$,$\{\psi_m\}$,$b$,$\sigma$,$r$}
        \State \textbf{Initialization}
        \State $n\gets 0$
        \State Precomputed normalization constants  $Z\gets \int_\Omega u_0(\bx)\intd \bx$  (either numerically or analytically).
        \State Generate $N$ initial particle positions set $S_0=\{\bX_{i,0}\}$ according to the distributions $\rho_0=\frac{u_0}{Z}$ and using the Metropolis--Hastings algorithm.
        \State \textbf{Interpolate (fixed-$N$):}
        \State $\tilde\rho_0 \gets \textsc{Interpolate}(S_0,N,\mathcal M_K,\{\psi_m\})$
        \State Append $\tilde u_0 = Z\tilde \rho_0$ into $\mathcal U$
        \While {$n\tau<T$}
            \State \textbf{SDE propagation:}
            \State $S_n^*\gets \textsc{SDE\_Propagate}(S_n,\tau,b,\sigma)$
            \State \textbf{Birth--Death:}
            \State $S_{n+1}\gets \textsc{Birth\_Death}(S_{n},\tau,r)$
            \State \textbf{refresh fields:}
            \State $\tilde\rho_{n+1} \gets \textsc{Interpolate}(S_{n+1},N,\mathcal M_K,\{\psi_m\})$
            \State Append $\tilde u_{n+1} = Z\tilde \rho_{n+1}$ into $\mathcal U$
            \State $n\gets n+1$
        \EndWhile
        \State \Return $\mathcal U$
        \EndProcedure
    \end{algorithmic}
\end{algorithm}

\begin{remark}
Propositions \ref{prop:u_K_mu} and \ref{prop:u_u_K} present the numerical convergence analysis for a scalar ADR equation. Extension to the two-variable KS system requires additional regularity assumptions of  the chemotactic drift, at least up to a stopping time prior to blow-up. A full two-field analysis is beyond the scope of this work and is left for future study.    
\end{remark}

\section{Extension to PDE system}
\label{sec:extension}
To solve PDE systems like Equation \eqref{equ:KS}, we introduce two sets of particles $\bX_i(t)$ and $\bY_i(t)$. Therefore, empirical distribution $\rho_u(t,\intd \bx ) = \mu_u(t,\intd \bx) =  \frac{1}{N_{u,0}}\sum_{i=1}^{N_{u,n}}\delta_{\bX_n^i} $ and $\rho_v(t,\intd \bx ) = \mu_v(t,\intd \bx) =  \frac{1}{N_{v,0}}\sum_{i=1}^{N_{v,n}}\delta_{\bY_n^i}$ are used to approximate $u$ and $v$ separately, by 
\begin{equation}
    u(t,\bx) =  Z_u \rho_u(t,\bx), \quad v(t,\bx) =  Z_v \rho_v(t,\bx),
\end{equation}
where $Z_u = \int_\Omega u(0,\bx) \intd \bx$ and $Z_v = \int_\Omega v(0,\bx) \intd \bx$. Therefore, the dynamics of \eqref{equ:KS} can be written as 
\begin{equation}
    \begin{aligned}
        \partial_t \rho_u =  \mathcal A_u \rho_u + \mathcal B_u \rho_u,\\
        \partial_t \rho_v =  \mathcal A_v \rho_v + \mathcal B_v \rho_v
    \end{aligned}
\end{equation}
where $\mathcal A_u \rho = \nabla \cdot ( \nabla \rho - \frac{Z_v}{Z_u} \chi(Z_u\rho) \nabla \rho_v)$ , $\mathcal A_v \rho = \Delta \rho$, $\mathcal B_u \rho_u  = f_u(Z_u\rho_u,Z_v \rho_v)$ and $B_v \rho_v = f_v(Z_u\rho_u,Z_v \rho_v)$. Similarly, we can approximate the solution by operator splitting, i.e., 
\begin{equation}
    \begin{aligned}
        \rho_{u,n}^*  =  \exp(\mathcal A_u\tau) \rho_{u,n}, \quad 
         \rho_{v,n}^*  =  \exp(\mathcal A_v\tau) \rho_{v,n},  \\
         \rho_{u,n+1} =  \exp(\mathcal B_u\tau) \rho_{u,n}^*,  \quad 
         \rho_{v,n+1}  =  \exp(\mathcal B_v\tau) \rho_{v,n}^*,   \\
    \end{aligned}
\end{equation}
where by matching parameters, we have $a_u(\bx;\rho_u,\rho_v) =  - \frac{Z_v}{Z_u\rho_u} \chi(Z_u\rho_u) \nabla \rho_v $ and $a_v(\bx;\rho_u,\rho_v)= 0$. For the diffusion part, we have $\sigma_u(\bx;\rho_u) = \sigma_v(\bx;\rho_v) = 1$. For birth and death part, $\tilde r_u (\bx;\rho_u,\rho_v)=  \frac{f_u(Z_u\rho_u,Z_v \rho_v)}{\rho_u}  $ and $\tilde r_v (\bx;\rho_u,\rho_v )=  \frac{f_v(Z_u\rho_u,Z_v \rho_v)}{\rho_v}  $. The detailed algorithm is shown in  Algorithm \ref{alg:ks}.

\begin{algorithm}
\caption{Stochastic Branching for keller--Segel equation}\label{alg:ks}
\begin{algorithmic}
\Require Initial PDE data $u_0(\bx),\,v_0(\bx)$; domain descriptor $\Omega$
\Require Fixed initial population sizes $N_u, N_v$
\Require Time step $\tau>0$, horizon $T$; truncation $\mathcal M_K$; basis $\{\psi_m\}_{m\in\mathcal M_K}$
\Require Normalization constants $Z_u=\int_\Omega u_0\,\mathrm d\bx$, $Z_v=\int_\Omega v_0\,\mathrm d\bx$
\Ensure Discrete KS fields $\mathcal U=\{\tilde u_n\}_{n=0}^{T/\tau}$, $\mathcal V=\{\tilde v_n\}_{n=0}^{T/\tau}$
\Procedure{KS\_Stochastic\_Branching}{$u_0,v_0,N_u,N_v,\tau,\mathcal M_K,\{\psi_m\}$}
  \State \textbf{Initialization}
  \State $n\gets 0$
  \State Generate $N_u$,$N_v$ initial particle positions $S_{u,0}=\{\mathbf X^i_0\}$, $S_{v,0}=\{\mathbf Y^i_0\}$ via MH from $\rho_{u,0}=u_0/Z_u$.
  \State \textbf{Interpolate (fixed-$N$):}
  \State $\tilde\rho_{u,0} \gets \textsc{Interpolate}\big(S_{u,0},N_u,\mathcal M_K,\{\psi_m\}\big)$
  \State $\tilde\rho_{v,0} \gets \textsc{Interpolate}\big(S_{v,0},N_v,\mathcal M_K,\{\psi_m\}\big)$
  \State Append $\tilde u_0 = Z_u \tilde\rho_{u,0}$ to $\mathcal U$;\quad Append $\tilde v_0 = Z_v \tilde\rho_{v,0}$ to $\mathcal V$
  \While{$n\tau<T$}
    \State \textbf{SDE propagation}
    \State $b_u(\bx) \gets -\dfrac{Z_v}{Z_u\,\tilde\rho_{u,n}(\bx)}\,\chi\!\big(Z_u\tilde\rho_{u,n}(\bx)\big)\,\nabla\tilde\rho_{v,n}(\bx)$; $b_v(\bx) \gets 0$
    \State $S_{u,n,*} \gets \textsc{SDE\_Propagate}\big(S_{u,n},\tau,b_u,\sigma_u\big)$
    \State $S_{v,n,*} \gets \textsc{SDE\_Propagate}\big(S_{v,n},\tau,b_v,\sigma_v\big)$
    \State \textbf{Interpolate at starred positions (fixed-$N$)}
    \State $\tilde\rho_u^* \gets \textsc{Interpolate}\big(S_{u,n,*},N_u,\mathcal M_K,\{\psi_m\}\big)$
    \State $\tilde\rho_v^* \gets \textsc{Interpolate}\big(S_{v,n,*},N_v,\mathcal M_K,\{\psi_m\}\big)$
    \State \textbf{Birth--Death}
    \State $r_u(\bx) \gets \dfrac{f_u\!\big(Z_u\tilde\rho_u^*(\bx),\,Z_v\tilde\rho_v^*(\bx)\big)}{\tilde\rho_u^*(\bx)}$; $r_v(\bx) \gets \dfrac{f_v\!\big(Z_u\tilde\rho_u^*(\bx),\,Z_v\tilde\rho_v^*(\bx)\big)}{\tilde\rho_v^*(\bx)}$
    \State $S_{u,n+1} \gets \textsc{Birth\_Death}\big(S_{u,n,*},\tau,r_u\big)$
    \State $S_{v,n+1} \gets \textsc{Birth\_Death}\big(S_{v,n,*},\tau,r_v\big)$
    \State \textbf{Refresh fields for output (fixed-$N$)}
    \State $\tilde\rho_{u,n+1} \gets \textsc{Interpolate}\big(S_{u,n+1},N_u,\mathcal M_K,\{\psi_m\}\big)$
    \State $\tilde\rho_{v,n+1} \gets \textsc{Interpolate}\big(S_{v,n+1},N_v,\mathcal M_K,\{\psi_m\}\big)$
    \State Append $\tilde u_{n+1}=Z_u\tilde\rho_{u,n+1}$ to $\mathcal U$;\quad Append $\tilde v_{n+1}=Z_v\tilde\rho_{v,n+1}$ to $\mathcal V$
    \State $n\gets n+1$
  \EndWhile
  \State \Return $\mathcal U,\ \mathcal V$
\EndProcedure
\end{algorithmic}
\end{algorithm}

\section{Numerical Examples}
In this section, we present numerical examples in two-dimensional space to verify the accuracy of our method.

\textbf{Case 1:} We first consider the AC equation, a widely used model for phase separation dynamics, to verify the accuracy of our scheme in handling nonlinear reaction systems. The governing equation is:
\begin{equation} \label{eq:ac_pde}
    u_t = D\Delta u + u - u^3, \quad \bx \in \Omega, t > 0,
\end{equation}
with the diffusion coefficient $D=0.01$. The computational domain is set to $\Omega = [0, 2\pi]^2$ with periodic boundary conditions imposed on all boundaries. 

The initial condition is given by:
\begin{equation} \label{eq:ac_ic}
    u_0(\bx) = \sin^2(x_1) \cos^2(x_2), \quad \bx = (x_1, x_2) \in \Omega.
\end{equation}
We solve the system numerically using our proposed method with Fourier projection $K=10$ and number of particle $N=400000$ up to a final time $T=0.8$. 
Figure~\ref{fig:ac_solution} shows the snapshots of the numerical solution $\tilde{u}$ at different time instances and its comparison with the finite difference method.
\begin{figure}
    \centering
    \includegraphics[width=0.8\linewidth]{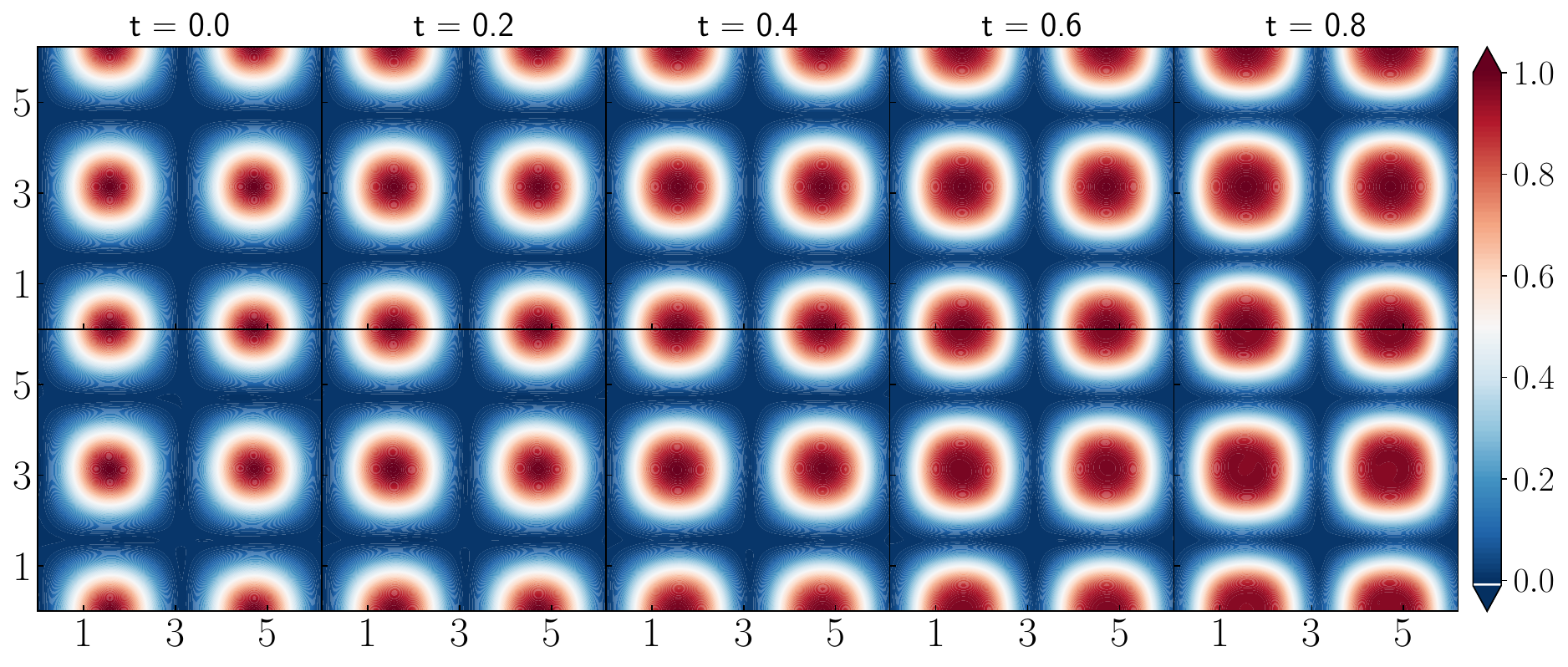}
    \caption{Numerical solution of the AC equation \eqref{eq:ac_pde} at different time snapshots ($T=0, 0.2, 0.4, 0.6, 0.8$). The top row displays the reference solution computed by a finite difference method, while the bottom row shows the results from our proposed particle method.}
    \label{fig:ac_solution}
\end{figure}

\textbf{Case 2:} Next, we consider the following KS equation in a periodic domain $\Omega = [0, 2\pi] \times [0, 2\pi]$:
\begin{equation} \label{equ:case1}
\begin{aligned}
&u_t - \nabla \cdot \left(\nabla u - u \nabla v \right) = 0  , &\quad \bx \in \Omega , t > 0, \\
&v_t - \Delta v = u - v , & \quad \bx \in \Omega, t > 0.
\end{aligned}
\end{equation}
The system is subject to periodic boundary conditions, with initial conditions given by:
$$
u_0 (\bx) =  \sin^2(\bx_1) \cos^2(\bx_2), \quad v_0(\bx) =  \cos(\bx_1) + \cos(\bx_2) + 2.
$$
The temporal evolution of the total mass of $v$ is characterized by
$$
M(t) = \int v(t,\bx) \, \intd \bx.
$$
It can be shown that the rate of change of the mass is given by
$$
\dot{M}(t) = -M(t) + \int u(x,0) \, dx = -M(t) + \pi^2.
$$
We present the numerical results for $M(t)$, comparing our method with the exact solution in Figure~\ref{fig:total_mass}.

\begin{figure}[h]
\centering
\includegraphics[width=0.8\linewidth]{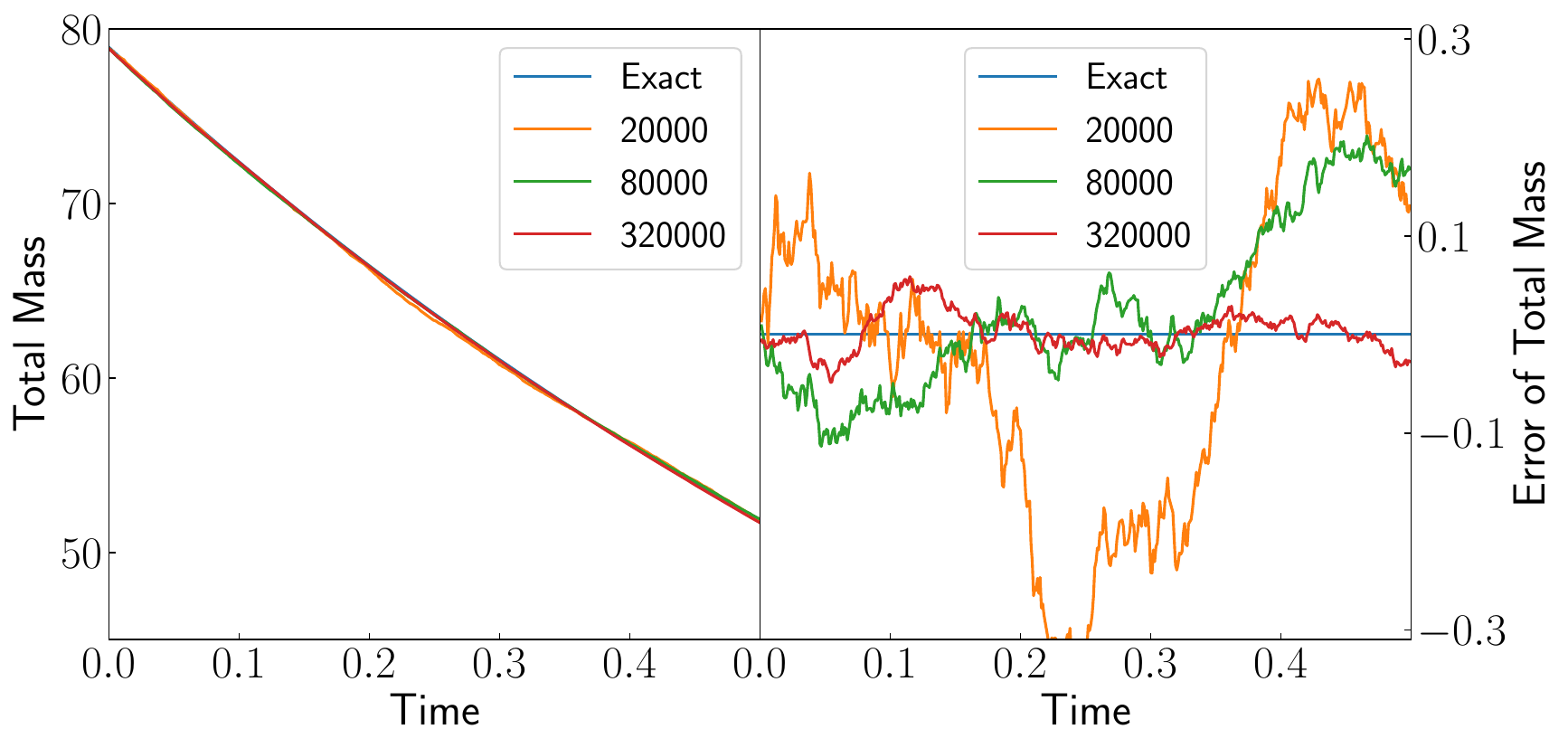}
\caption{Time evolution of the total mass of $v$ for different numbers of particles. The left panel shows the total mass over time, while the right panel displays the corresponding error in the total mass.}
\label{fig:total_mass}
\end{figure}

Additionally, we compare our particle-based method with finite difference method with a grid of $100 \times 100$ points. The result for $u(t)$ and $v(t)$ at $t = 0.0, 0.1, 0.2, 0.3, 0.4$ are shown in Figures~\ref{fig:fd_comparison_u} and~\ref{fig:fd_comparison_v}, respectively. Figure~\ref{fig:convergence_error} illustrates the reduction in relative $L^2$ error as the number of samples increases, showing a clear $1/\sqrt{N}$ convergence rate.

\begin{figure}[h]
\centering
\includegraphics[width=\linewidth]{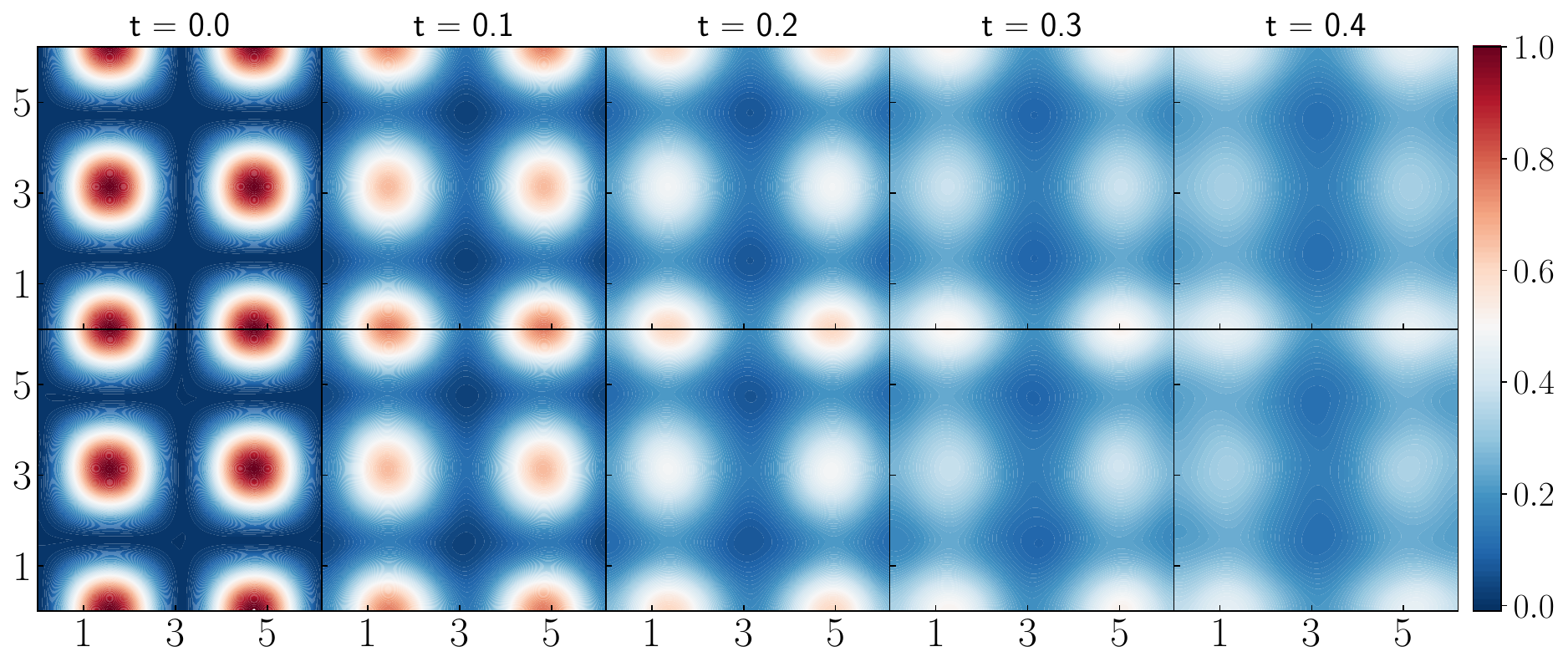}
\caption{Comparison of the numerical approximation of $u$ obtained from the finite difference method with a $100 \times 100$ grid and the present particle-based method at $t = 0.0, 0.1, 0.2, 0.3, 0.4$.}
\label{fig:fd_comparison_u}
\end{figure}

\begin{figure}[h]
\centering
\includegraphics[width=\linewidth]{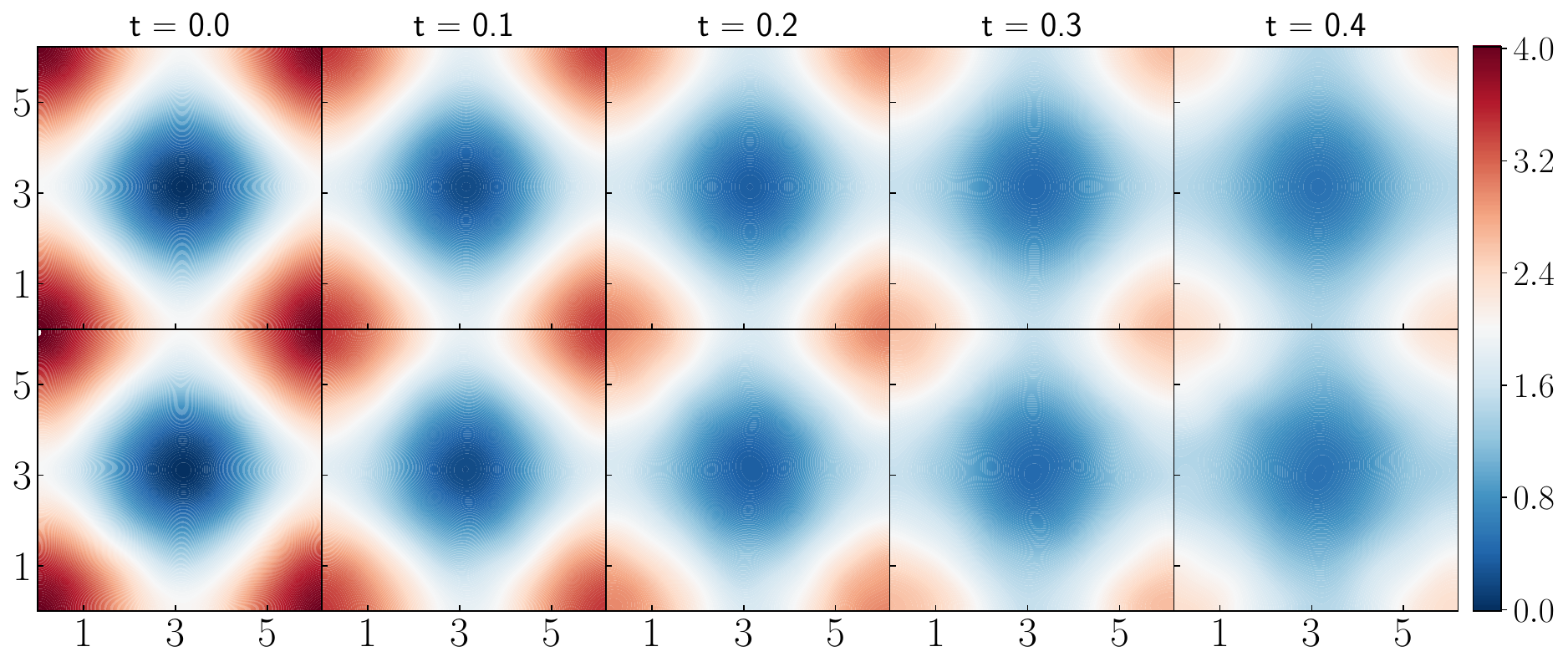}
\caption{Comparison of the numerical approximation of $v$ obtained from the finite difference method with a $100 \times 100$ grid and the present particle-based method at $t = 0.0, 0.1, 0.2, 0.3, 0.4$.}
\label{fig:fd_comparison_v}
\end{figure}

\begin{figure}[h]
\centering
\includegraphics[width=0.8\linewidth]{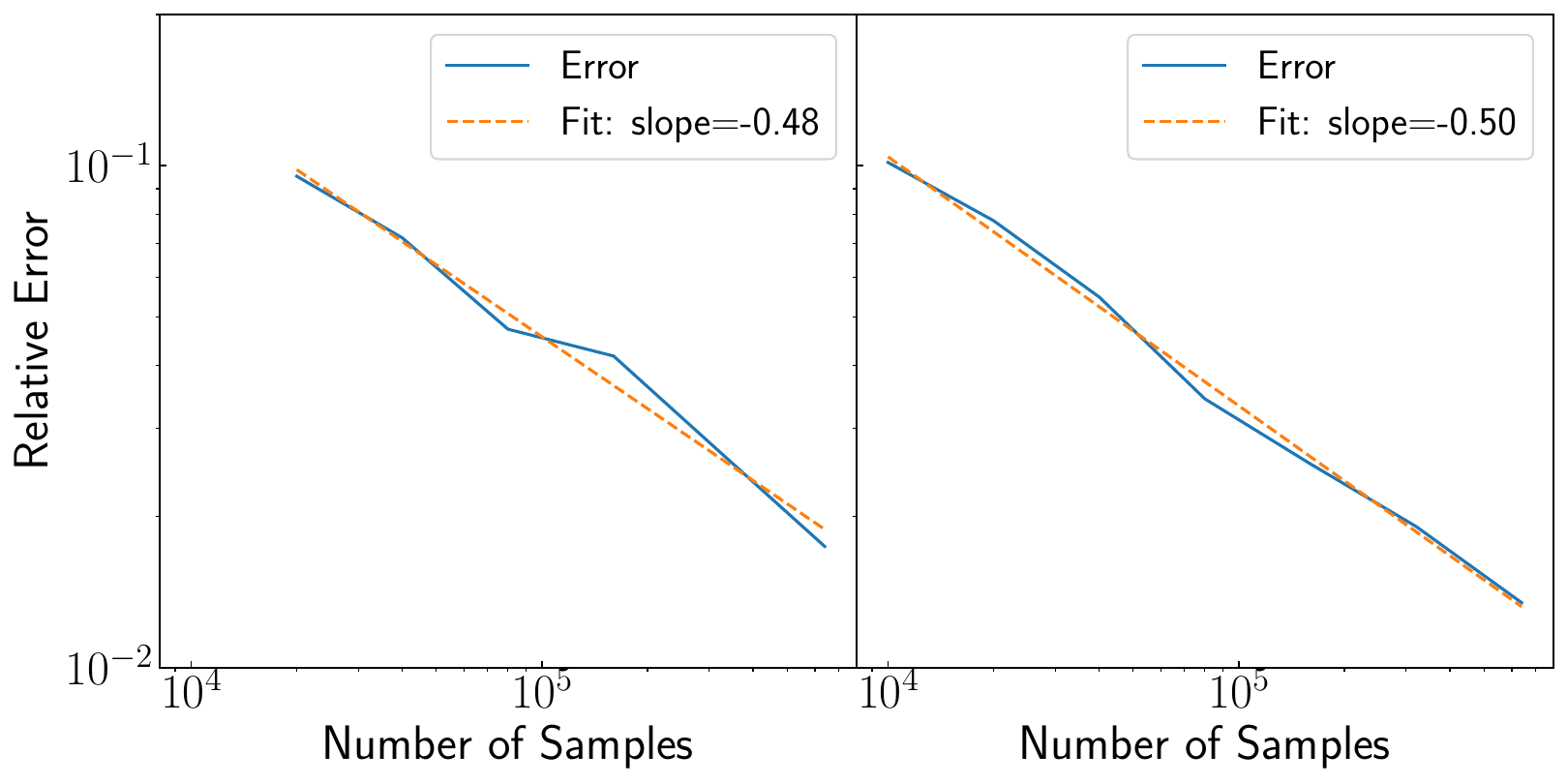}
\caption{Relative $L^2$ error of $u$ and $v$ with respect to the number of samples $N_0$. The reference solution is obtained with $N_0 = 320,000$.}
\label{fig:convergence_error}
\end{figure}

\textbf{Case 3:} Furthermore, we solve the system \eqref{equ:case1} with the initial conditions:
$$
u_0(\bx) = 840 \exp\left(-84 (\bx_1^2 + \bx_2^2)\right), \quad v_0(\bx) = 420 \exp\left(-42(\bx_1^2 + \bx_2^2)\right).
$$
Figure~\ref{fig:case2_u} shows the numerical solution $u$, which undergoes a blow-up phenomenon that is generally difficult to capture using the standard grid discretization. To assess the performance of our method, we compare its results against a finite difference simulation. Three resolution levels are analyzed: the particle method with $N = 40,000$, $160,000$, and $640,000$ is compared to the finite difference method with $200 \times 200$, $400 \times 400$, and $800 \times 800$ mesh points. 
Figure \ref{fig:comparison} shows the numerical solution at $t=5 \times 10^{-5}$. The present particle method enables us to accurately capture the aggregation dynamics and is robust with respect to the number of particles. In contrast, the solution obtained from the finite difference method exhibits pronounced mesh size sensitivity near the aggregation. While the performance of the grid-based approach can be improved by applying more sophisticated spatial discretization such as local discontinuous Galerkin method~\cite{li2017local}, mesh refinement is often required to track the aggregation. On the other hand, the present particle method is natural for modeling such strong nonlinear behaviors without using the adaptive mesh refinement.


\begin{figure}
\centering
\includegraphics[width=0.8\linewidth]{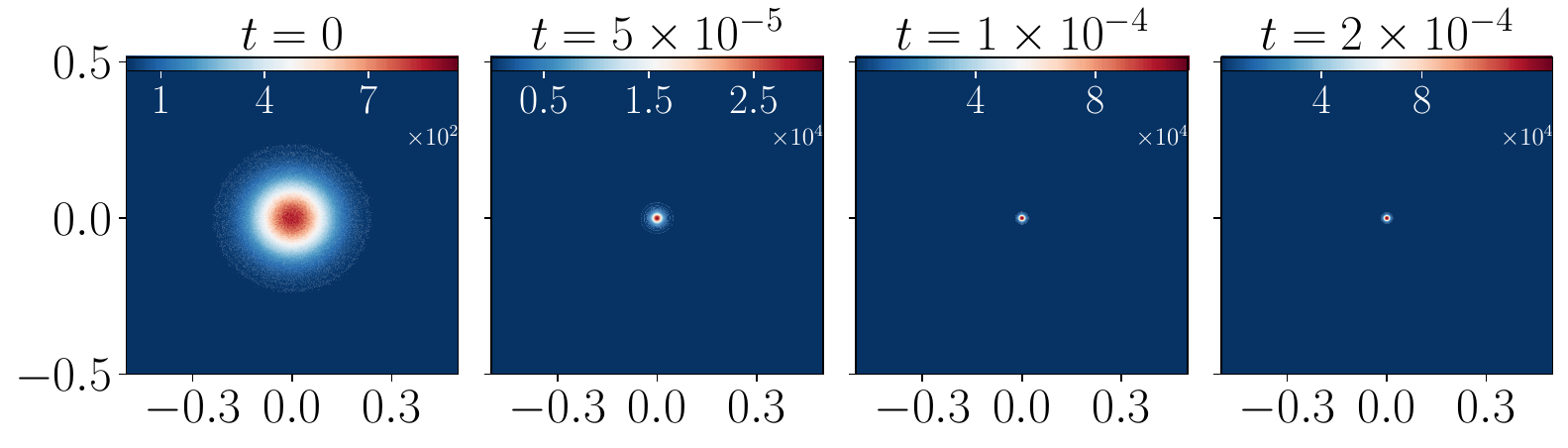}
\caption{Numerical approximations of $u$ at $t = 0, 5 \times 10^{-5}, 10^{-4}$, and $1.5 \times 10^{-4}$.}
\label{fig:case2_u}
\end{figure}
\begin{figure}
\centering
\includegraphics[width=0.85\linewidth]{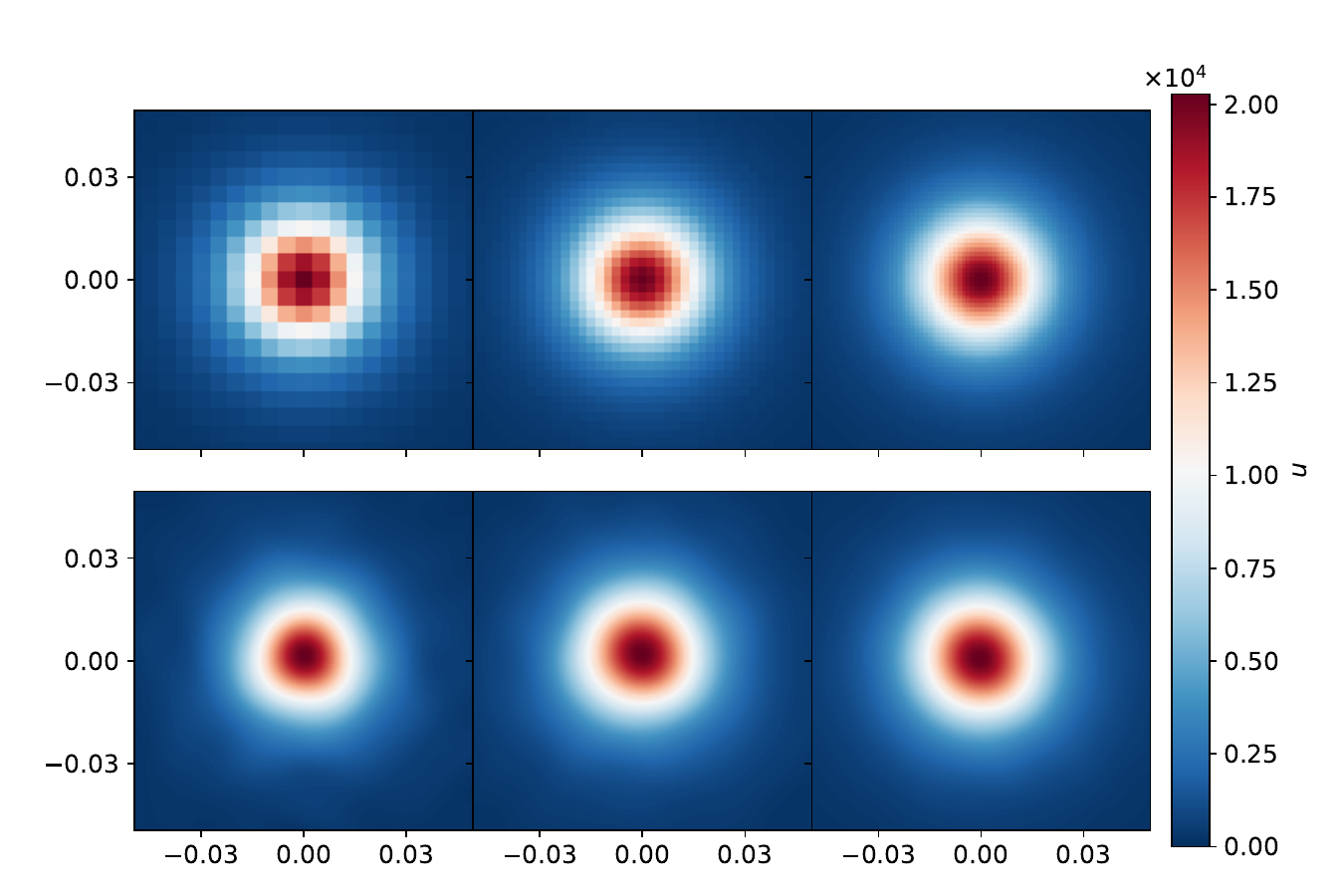}
\caption{Comparison of the numerical solution $u$ from the finite difference method and the present particle method at $t=5\times 10^{-5}$.}
\label{fig:comparison}
\end{figure}

\textbf{Case 4:} Finally, we examine a fully nonlinear case with logistic growth for the cells, where the interaction between the particles is not known a priori. The system is given by:
\begin{equation}
\label{equ:nonliear}
\begin{aligned}
&u_t - \nabla \cdot \left(\nabla u - \frac{4u}{1 + u^2} \nabla v \right) = u(1 - u), & \quad \bx \in \Omega, t > 0, \\
&v_t - \Delta v = u - v, & \quad \bx \in \Omega, t > 0.
\end{aligned}
\end{equation}
The theoretical results of this model can be found in \cite{osaki2002exponential, wang2007classical, wrzosek2004global}. The initial conditions are specified as:
$$
    \begin{aligned}
        u_0(\bx) =&  \exp\left( (\bx_1 - 0.5\pi)^2 + (\bx_2 - 0.5\pi)^2 \right) + \exp\left( (\bx_1 - \pi)^2 + (\bx_2 - \pi)^2 \right) \\
        &+ \exp\left( (\bx_1 - 1.5\pi)^2 + (\bx_2 - 1.5\pi)^2 \right), \\
        v_0(\bx) =&  \exp\left( (\bx_1 - 0.4\pi)^2 + (\bx_2 - 0.4\pi)^2 \right) + \exp\left( (\bx_1 - 0.8\pi)^2 + (\bx_2 - 0.8\pi)^2 \right) \\
        &+ \exp\left( (\bx_1 - 1.2\pi)^2 + (\bx_2 - 1.2\pi)^2 \right) + \exp\left( (\bx_1 - 1.6\pi)^2 + (\bx_2 - 1.6\pi)^2 \right).
    \end{aligned}
$$

With the non-conservative form, both the birth and death of cells, as well as the chemical concentrations $f_u$ and $f_v$, are non-zero. The particle method proposed in \cite{wang2025novel} only accounts for advection and diffusion and therefore has limited capability in addressing these effects.  In contrast, the present method can naturally handle these additional complexities in the non-conservative form. 
Figures~\ref{fig:case3_u} and~\ref{fig:case3_v} show the time evolution of the solutions $u$ and $v$ obtained using both the particle-based and the finite difference methods.

\begin{figure}
\centering
\includegraphics[width=0.8\linewidth]{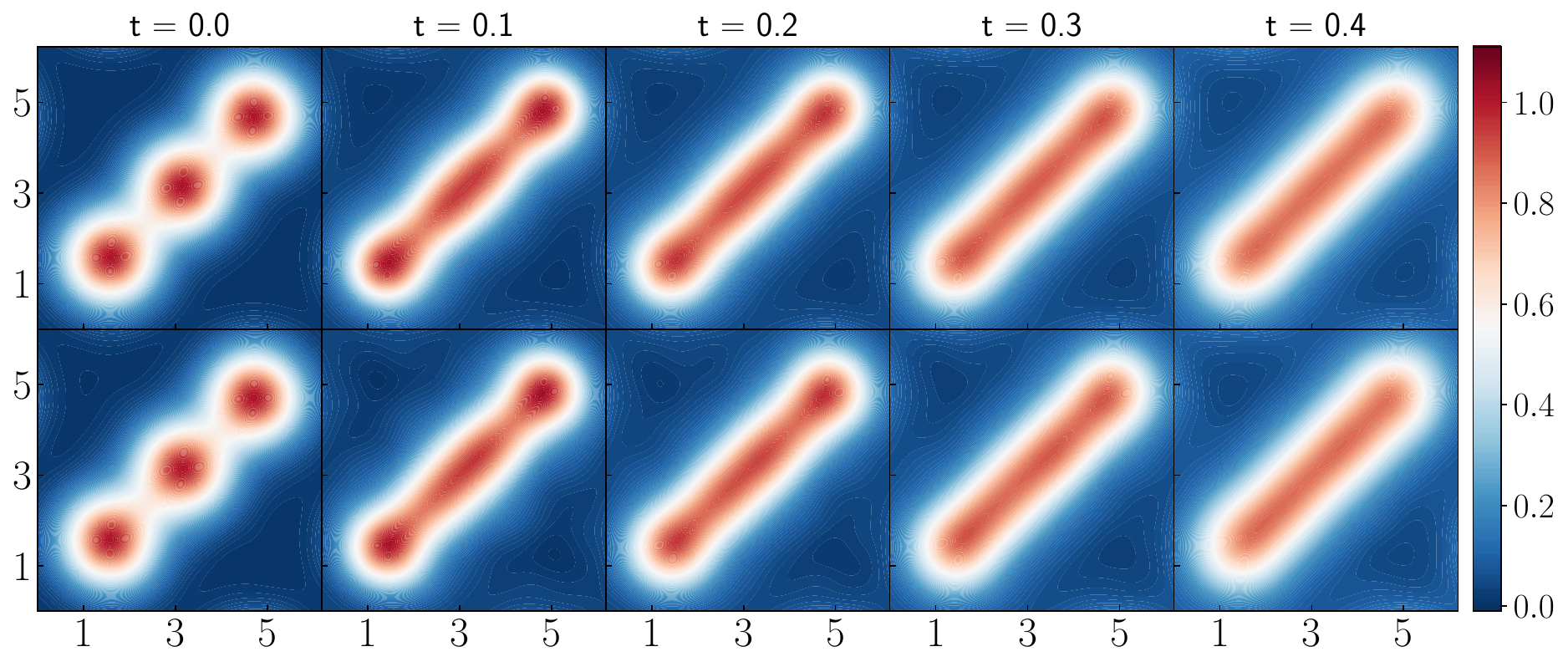}
\caption{Time evolution of the numerical solution $u$ computed using the particle-based and the finite difference methods.}
\label{fig:case3_u}
\end{figure}

\begin{figure}
\centering
\includegraphics[width=0.8\linewidth]{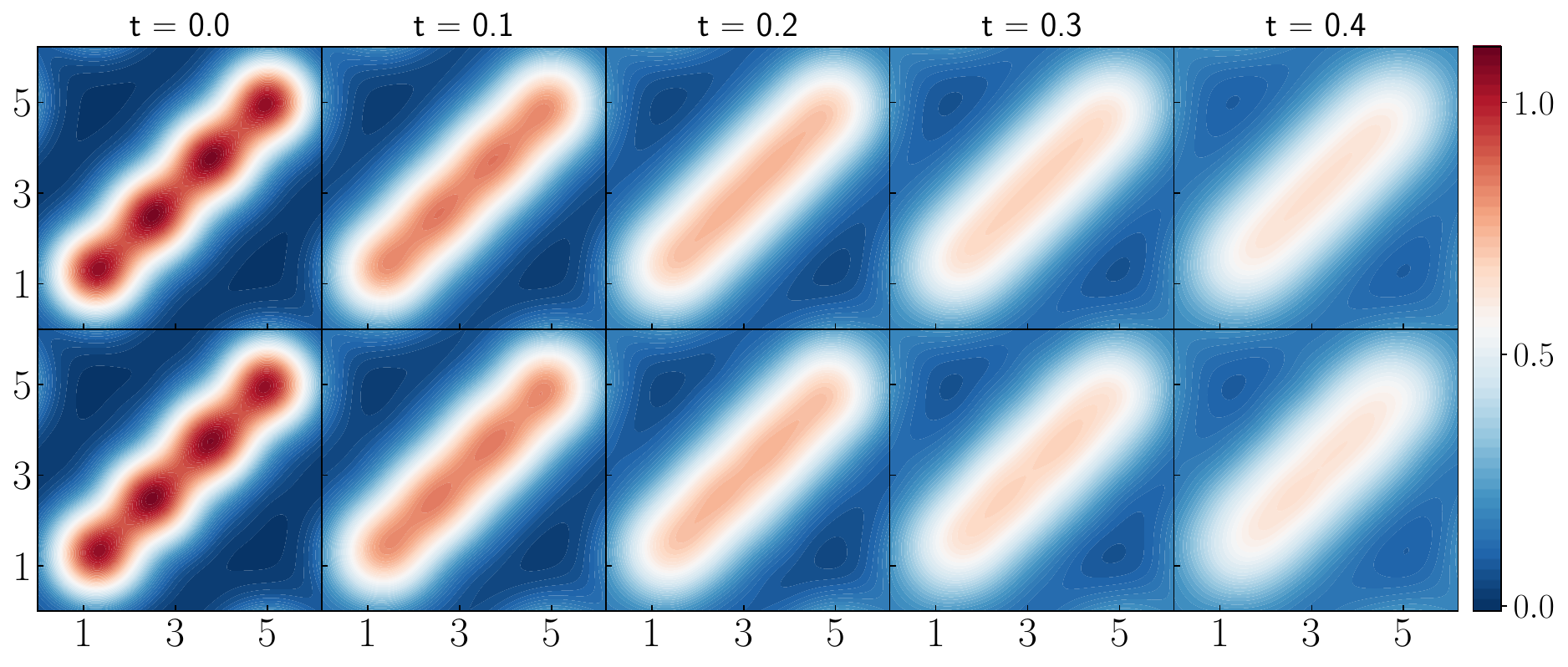}
\caption{Time evolution of the numerical solution $v$ computed using the particle-based and the finite difference methods.}
\label{fig:case3_v}
\end{figure}

\section{Discussion}
This work presents a stochastic particle-based approach for solving nonlinear non-conservative advection–diffusion–reaction equations. 
The evolution is split into two steps: an advection–diffusion step implemented by particle dynamics, and a local reaction step represented by a stochastic branching birth–death mechanism that provides a consistent temporal discretization of the underlying reaction dynamics. 
Unlike traditional particle or Lagrangian schemes, this formulation enables the present method to naturally capture the non-conservative evolution without an explicit form of the microscopic physical law. The resulting scheme preserves nonnegativity, captures concentration and blow-up robustly, and achieves an unbiased expectation with $O(1/N)$ Monte Carlo variance. Importantly, the algorithm is inherently parallel and does not require grid-based discretization.
Numerical results of the AC and the KS equations show the generality and robustness. In particular, the present method accurately captures the nonlinear behaviors such as aggregation and mass evolution of the KS equation without requiring adaptive mesh refinement, and is applicable to a broad class of transport and interaction problems. 

Looking forward, the current work lays the foundation for several promising extensions. From a numerical perspective, the present particle-based formulation is particularly well-suited for high-dimensional PDE problems, where grid-based methods often show limitions. This calls for efficient interpolation of empirical particle measures. In particular, the low-rank spectral approximations and adaptive reconstruction techniques offer promising directions ~\cite{tang2025solving, Peng_Yang_tensor_arxiv_2024}. Additionally, the present method can be extended for more general boundary conditions such as Dirichlet, Neumann, and Robin types, where proper stochastic killing, reflection, or mixed rules need to be developed at the particle level. Finally, while we provide convergence results for the scalar case, theoretical analysis of fully coupled multi-species systems requires further investigation. We leave these for future studies. 



\section*{Acknowledgments}
The work is supported in part by the National Science Foundation under Grant DMS-2143739 and the ACCESS program through allocation MTH210005. 
The authors also acknowledge the support from the Institute for Cyber-Enabled Research at Michigan State University.







\end{document}